\documentclass[reqno]{amsart}
\usepackage[left=1in,right=1in,top=1in,bottom=1in]{geometry}
\usepackage{tikz}
\usetikzlibrary{shapes,decorations,calc,arrows}
\usepackage[foot]{amsaddr}
\usepackage{amsthm}
\usepackage{amscd}
\usepackage{amsfonts}
\usepackage{amsmath}
\usepackage{amssymb}
\usepackage{mathrsfs}
\usepackage{multirow}
\usepackage{verbatim}
\usepackage{url}
\usepackage[hidelinks]{hyperref}
\usepackage{graphicx}
\usepackage{cite}
\usepackage{fancyhdr}
\usepackage{yfonts}
\usepackage{setspace}
\usepackage{titlesec}
\usepackage{enumitem}

\usepackage{ dsfont }

\usepackage{tocloft} 

\titleformat{\section}[hang]
{\normalfont\Large\bfseries}
{\thesection.}{0.5em}{}

\titlespacing*{\section}{0pc}{2pc}{0.25pc}

\titleformat{\subsection}[runin]
{\normalfont\large\bfseries}
{\thesubsection}{0.5em}{}

\titlespacing{\subsection}{0pc}{1.5pc}{0.5pc}

\newcommand{\Aut}{\text{Aut}}

\newcommand{\N}{\mathbb{N}}
\newcommand{\Z}{\mathbb{Z}}

\newcommand{\R}{\mathbb{R}}
\newcommand{\C}{\mathbb{C}}
\renewcommand{\H}{\mathcal{H}}

\newcommand{\vphi}{\varphi}

\newcommand{\eig}{\text{eig}}

\newcommand{\Tr}{\text{Tr}}
\newcommand{\tr}{\text{tr}}

\DeclareMathOperator{\conv}{conv}

\newcommand{\<}{\left\langle}
\renewcommand{\>}{\right\rangle}
\renewcommand{\Re}[1]{\text{Re}\ #1}

\newcommand{\Ad}[1]{\text{Ad}\left(#1\right)}

\newcommand{\E}{\mathcal{E}}

\newcommand{\Sd}{\text{Sd}}

\newtheorem{thm}{Theorem}[section]
\newtheorem{thmalpha}{Theorem}

\newtheorem{coralpha}[thmalpha]{Corollary}

\newtheorem{lem}[thm]{Lemma}
\newtheorem*{lem*}{Lemma}

\theoremstyle{definition}

\title{\textbf{Uniqueness of extremal almost periodic states on the injective type $\mathrm{III}_1$ factor}}
\author{Michael Hartglass$^\circ$}
\address{$^\circ$Department of Mathematics and Computer Science, Santa Clara University\hfill \url{mhartglass@scu.edu}}
\author{Brent Nelson$^\bullet$}
\address{$^\bullet$Department of Mathematics, Michigan State University \hfill \url{brent@math.msu.edu}}
\date{}

\begin{document}

\maketitle

\begin{abstract}
Let $R_\infty$ denote the Araki--Woods factor---the unique separable injective type $\mathrm{III}_1$ factor. For extremal almost periodic states $\vphi, \psi\in (R_\infty)_*$, we show that if $\Delta_\vphi$ and $\Delta_\psi$ have the same point spectrum then $\psi = \vphi\circ \alpha$ for some $\alpha\in \Aut(R_\infty)$. Consequently, the extremal almost periodic states on $R_\infty$ are parameterized by countable dense subgroups of $\R_+$, up to precomposition by automorphisms. As an application, we show that KMS states for generalized gauge actions on Cuntz algebras agree (up to an automorphism) with tensor products of Powers states on their von Neumann completions.
\end{abstract}


\section*{Introduction}

Using approximations by completely positive maps, Haagerup provided a new proof in \cite{Haa89} of the uniqueness of the separable injective type $\mathrm{III}_\lambda$ factors for $0<\lambda<1$. This uniqueness was first established by Connes in his celebrated work \cite{Con76}, but a novel consequence of Haagerup's proof was the uniqueness of their $\frac{2\pi}{\log(1/\lambda)}$-periodic states. For $0 < \lambda \leq 1$ denote
    \begin{align}\label{eqn:def_of_Powers_factors}
        (R_\lambda,\vphi_\lambda) := \bigotimes_{n\in \N} (M_2(\C),\phi_\lambda)
    \end{align}
where
    \[
        \phi_\lambda(x) := \Tr\left[\left(\begin{array}{cc} \frac{1}{1+\lambda} & 0 \\ 0 & \frac{\lambda}{1+\lambda} \end{array}\right)x \right].
    \]
These are separable injective factors of type $\mathrm{III}_\lambda$ for $0<\lambda< 1$ (also known as \emph{Powers' factors} having been first studied in \cite{Pow67}), and type $\mathrm{II}_1$ for $\lambda=1$. The associated states are $\frac{2\pi}{\log(1/\lambda)}$-periodic for $0<\lambda<1$, while $\tau:=\vphi_1$ is of course the unique tracial state on $R:=R_1$. For $(N,\vphi)$ a separable injective type $\mathrm{III}_\lambda$ factor equipped with a $\frac{2\pi}{\log(1/\lambda)}$-periodic state and $0<\lambda <1$, Haagerup established a state-preserving isomorphism $(N,\vphi)\cong (R_\lambda,\vphi_\lambda)$, and consequently any pair of $\frac{2\pi}{\log(1/\lambda)}$-periodic states on $R_\lambda$ differ only by precomposition with an automorphism (see \cite[Theorem 6.1 and Corollary 6.5]{Haa89}). Haagerup also established the uniqueness of the separable injective type $\mathrm{III}_1$ factor in \cite{Haa87} using a reduction due to Connes in \cite{Con85}, and a second proof was published posthumously in \cite{Haa16} that more closely followed the approach in \cite{Haa89}. However, neither of these proofs provide a \emph{state preserving} isomorphism and consequently the analogue of \cite[Corollary 6.5]{Haa89} has been missing in the type $\mathrm{III}_1$ setting until now. In this article we prove such an analogue, where periodic states are replaced with so-called extremal almost periodic states. 

In \cite{Con74}, Connes defined an \emph{almost periodic} weight as a faithful normal semifinite weight $\vphi$ whose modular operator $\Delta_\vphi$ is diagonalizable:
    \[
        \Delta_\vphi = \sum_{\lambda\in \Sd(\vphi)} \lambda 1_{\{\lambda\}}(\Delta_\vphi),
    \]
where we denote the point spectrum of $\Delta_\vphi$ by $\Sd(\vphi)$. An \emph{extremal} weight is one whose centralizer subalgebra $M^\vphi := \{x\in M\colon \sigma_t^{\vphi}(x) = x \ \forall t\in \R\}$ is a factor. Almost periodic weights have the property that $(M^\vphi)'\cap M = (M^\vphi)'\cap M^\vphi$ (see \cite[Theorem 10]{Con72}), and so extremal almost periodic weights only occur on factors. Additionally, $\Sd(\vphi)$ for an extremal almost periodic weight $\vphi$ is a subgroup of $\R_+$, the group of positive reals $(0,\infty)$ under multiplication (see, for example, \cite[Lemma 4.9]{Dyk97}). The states $\vphi_\lambda$ in (\ref{eqn:def_of_Powers_factors}) are extremal almost periodic: for $\lambda<1$ this follows from \cite[Theorem 4.2.6]{Con73}; and $\tau$ is tracial so that $\Delta_{\tau}=1$ and $R^{\tau}=R$. Aperiodic examples of extremal almost periodic weights necessarily occur on type $\mathrm{III}_1$ factors, and, while not all type $\mathrm{III}_1$ factors admit such weights (see \cite[Corollary 5.3]{Con74}), the unique separable injective type $\mathrm{III}_1$ factor $R_\infty$ (also known as the \emph{Araki--Woods factor}) has many. To see this, extend the notation in (\ref{eqn:def_of_Powers_factors}) to $\lambda>1$ by letting $\phi_\lambda:=\phi_{1/\lambda}$ and $(R_\lambda,\vphi_\lambda):= (R_{1/\lambda},\vphi_{1/\lambda})$. For a countable subgroup $\Lambda\leq \R_+$, denote
    \begin{align}\label{eqn:def_of_R_Lambda}
        (R_\Lambda,\vphi_\Lambda):= \bigotimes_{\lambda\in \Lambda} (R_\lambda,\varphi_\lambda).
    \end{align}
If $\Lambda$ is of the form $\lambda^\Z$ for some $0<\lambda \leq 1$, then $(R_\Lambda,\vphi_\Lambda)\cong (R_\lambda,\vphi_\lambda)$ thanks to the property $(R_\lambda\bar\otimes R_\lambda,\vphi_\lambda\otimes \vphi_\lambda) = (R_\lambda,\vphi_\lambda)$. Otherwise, $\Lambda$ is dense in $\R_+$ and $R_\Lambda \cong R_\infty$. Moreover, $\vphi_\Lambda$ is an extremal almost periodic state (see, for example, \cite[Lemma 5.2.(i)]{HI22}). Our main result is the following:

\begin{thmalpha}\label{thm:A}
Let $N$ be a diffuse separable injective factor equipped with an extremal almost periodic state $\vphi$. Then 
    \[
        (N,\vphi)\cong (R_\Lambda,\vphi_\Lambda)
    \]
where $\Lambda = \Sd(\vphi)$.
\end{thmalpha}

If $\Lambda = \lambda^{\Z}$ for some $0< \lambda \leq 1$ then the above is simply either \cite[Theorem 6.1]{Haa89} for $\lambda <1$ or \cite[Theorem 1]{Con76} for $\lambda=1$. Thus the novelty of the above theorem lies only in the case that $\Lambda$ is dense in $\R_+$, but nevertheless our proof simultaneously covers all cases. The proof of Theorem~\ref{thm:A} very closely follows Haagerup's proof of \cite[Theorem 6.1]{Haa89} and can be found in Section~\ref{sec:proof_thmA}.  Roughly, the strategy is to approximate a countable generating set for $N$ by a sequence of finite dimensional subfactors that when equipped with the restriction of $\vphi$ are isomorphic to a finite dimensional tensor factor of $(R_\Lambda, \vphi_\Lambda)$. These subfactors are constructed using approximate completely positive factorizations of the identity map through matrix algebras, whose existence is established in Section~\ref{sec:modular_cp} (paralleling \cite[Section 5]{Haa89}). Notably, Lemma~\ref{lem:approximate_cp_embeddings_into_R} marks a technical deviation from Haagerup's ideas, though it still essentially corresponds to \cite[Lemma 5.3]{Haa89}.
These completely positive maps also respect the modular theory of $\vphi$, a property which is needed to apply the results in Section~\ref{sec:approximate_Kraus_representations} (paralleling \cite[Section 4]{Haa89}). These results, namely Theorem~\ref{thm:unital_but_infinite_Kraus_approximation}, approximate the completely positive maps by $x\mapsto \sum_{n=1}^\infty a_n^* x a_n$ for $(a_n)_{n\in \N}\subset N^\vphi$ a partition of unity, which in turn allows one to further approximate the maps by an inner automorphism using the notion of $\delta$-related tuples from \cite[Section 2]{Haa89}. Despite how similar our results are to those in \cite{Haa89}, working with non-cyclic subgroups of $\R_+$ almost always prevents us from simply citing the corresponding results. So instead the details are presented in full and the reader should find the paper mostly self-contained.

We also remark that if one could directly argue that a separable injective type $\mathrm{III}_1$ factor admits at least one extremal almost periodic state, then Theorem~\ref{thm:A} would yield a new proof of the uniqueness of such factors. However, this remains out of reach. In fact, Haagerup's approach in \cite{Haa87} specifically avoids this by working instead with dominant weights, which are known to always exist for properly infinite factors by \cite[Chapter II]{CT77}.

As an immediate consequence of Theorem~\ref{thm:A} we have the following.

\begin{coralpha}\label{cor:B}
If $\vphi$ and $\psi$ are extremal almost periodic states on $R_\infty$ with $\Sd(\vphi)=\Sd(\psi)$, then there exists $\alpha\in \Aut(R_\infty)$ satisfying $\psi=\vphi\circ \alpha$.
\end{coralpha}

This corollary combined with the existence of the extremal almost periodic states $\vphi_\Lambda$ in (\ref{eqn:def_of_R_Lambda}) implies there is a one-to-one correspondence between the collection of countable dense subgroups of $\R_+$ and the equivalence classes of extremal almost periodic states on $R_\infty$ under precomposition by automorphisms.

Due to its uniqueness, $R_\infty$ arises from a variety of constructions, and so our results can be used to recognize extremal almost periodic states on these constructions as some $\vphi_\Lambda$. Our last main result is an example of this involving the Cuntz algebra.
If $J$ is any finite or countably infinite set, then recall that the Cuntz algebra, $\mathcal{O}_{J}$, is the universal C*-algebra generated by a family of isometries $\{v_{j}\}_{j \in J}$ with mutually orthogonal ranges, and if $J$ is finite it is further required that $\sum_{j \in J} v_{j}v_{j}^{*} = 1$. Let $\mu\colon J\to (0,1)$ be a function with the property that $\sum_{j \in J} \mu(j) = 1$. Define a gauge action $\R \overset{\sigma}{\curvearrowright}\mathcal{O}_{J}$ by
    \[
        \sigma_{t}(v_{j}) = \mu(j)^{it}v_{j},
    \]
for each $t\in \R$ and $j\in J$. It is known that a unique KMS state $\vphi$ exists for such an action (see, for example, \cite{Tho17}). Letting $\pi_\vphi$ be its GNS representation, we will abuse notation slightly and let $\vphi$ also denote its extension to $\pi_\vphi(\mathcal{O}_J)''$ as a faithful normal state. Denote by $\Lambda$ the multiplicative subgroup of $\R_+$ generated by $\mu(J)$.  It has been shown by Izumi for $J$ finite \cite{Izu93} and Thomsen for $J$ infinite \cite{Tho19} that $\pi_\vphi(\mathcal{O}_J)''$ is a separable injective factor of type $\mathrm{III}_{\lambda}$ if $\Lambda=\lambda^\Z$ for some $\lambda \in (0, 1)$ and of type $\mathrm{III}_{1}$ otherwise. In the former case, \cite[Theorem 6.1]{Haa89} implies $(\pi_\vphi(\mathcal{O}_J)'',\vphi)\cong (R_\lambda,\vphi_\lambda)$ since $\vphi$ is necessarily $\frac{2\pi}{\log(1/\lambda)}$-periodic. 
In the general case we argue in Section~\ref{section:Cuntz_algebras} that Theorem~\ref{thm:A} can be applied to obtain the following.

\begin{thmalpha}\label{thm:C}
With the above notation, $\vphi$ is an extremal almost periodic state on $\pi_\vphi(\mathcal{O}_J)''$ and there is a state preserving isomorphism $(\pi_\vphi(\mathcal{O}_J)'',\vphi)\cong (R_\Lambda, \vphi_\Lambda)$.
\end{thmalpha}

\subsection*{Acknowledgements}

The authors thank Cyril Houdayer for suggesting this problem. The second author was supported by NSF grant DMS-2247047.


\section{Preliminaries}

We say a von Neumann algebra $M$ is \emph{separable} if it is separable in the weak* topology induced by $M\cong (M_*)^*$; equivalently, if $M_*$ is separable as a Banach space. We write $N\leq M$ when $N$ is a unital von Neumann subalgebra of $M$. For a family of subalgebras $\{N_i\leq M\colon i\in I\}$, we write $\bigvee_{i\in I} N_i$ for the von Neumann algebra generated by $\bigcup_{i\in I} N_i$. For a state $\vphi$ on $M$ we define the inner product
    \[
        \<x,y\>_\vphi = \vphi(y^*x) \qquad x,y\in M,
    \]
and denote its associated norm by $\|x\|_\vphi$.

We assume the reader has some familiarity with weights and modular theory, but the complete details can be found in \cite{Tak03} (see also \cite[Section 1]{GGLN}).  For a faithful normal semifinite weight on $M$, its \emph{modular automorphism group} refers to the action $\R\overset{\sigma^\vphi}{\curvearrowright} M$ defined by
    \[
        \sigma_t^{\vphi}(x) := \Delta_\vphi^{it} x \Delta_\vphi^{-it}.
    \]
We say a subalgebra $N\leq M$ is \emph{$\vphi$-invariant} if $\vphi|_N$ is semifinite and $\sigma_t^{\vphi}(N) = N$ for all $t\in \R$, which is equivalent to the existence of a faithful normal $\vphi$-invariant conditional expectation from $M$ onto $N$ (see \cite[Theorem IX.4.2]{Tak03}). For $\lambda\in \R_+$ we denote
    \[
        M^{(\vphi,\lambda)} := \{x\in M\colon \sigma_t^{\vphi}(x) = \lambda^{it} x\ \forall t\in \R\},
    \]
so that $M^{(\vphi,1)}=M^\vphi$. We call such elements \emph{eigenoperators} of $\vphi$. We also denote
    \[
        M^{(\vphi,\eig)}:= \text{span}\left( \bigcup_{\lambda\in \R_+} M^{(\vphi,\lambda)} \right),
    \]
which is a unital $*$-subalgebra of $M$, and it is weak* dense when $\vphi$ is almost periodic (see \cite[Lemma 1.4]{GGLN}). 

Given an almost periodic weight $\vphi$ on $M$, let $\Lambda \leq \R_+$ be the group generated by $\Sd(\vphi)$. Equipping $\Lambda$ with the discrete topology, let $G$ denote its compact dual group. Identifying the dual of $\R_+$ with $\R$ via the pairing $(s,t) = s^{it}$, the \emph{transpose} of the inclusion map $\iota\colon \Lambda\to \R_+$ gives a continuous group homomorphsim $\hat{\iota}\colon \R\to G$ determined by
    \[
        (\hat{\iota}(t)|\lambda) = (\iota(\lambda) | t) = \lambda^{it}.
    \]
The action $\R\overset{\sigma^\vphi}{\curvearrowright} M$ given by the modular automorphism group of $\vphi$ determines an action $G\overset{\alpha}{\curvearrowright} M$ satisfying $\alpha_{\hat{\iota}(t)} = \sigma_t^\vphi$ for all $t\in \R$. In particular, $M^\alpha = M^\vphi$ and more generally for $x\in M^{(\vphi,\lambda)}$ one has $\alpha_s(x) = (\lambda| s) x$ for all $s\in G$. We call this action of $G$ the \emph{point modular extension} of $\R\overset{\sigma^\vphi}{\curvearrowright} M$. For each $\lambda\in \Lambda$, there is a normal $M^\vphi$-$M^\vphi$-bimodular idempotent map $\E_{(\vphi,\lambda)}\colon M\to M^{(\vphi,\lambda)}$ defined by
    \[
        \E_{(\vphi,\lambda)}(x):= \int_G \overline{(\lambda| s)} \alpha_s(x)\ d\mu(s),
    \]
where $\mu$ is a normalized Haar measure on $G$. In particular, $\E_\vphi:=\E_{(\vphi,1)}$ is the unique $\vphi$-preserving faithful normal conditional expectation onto the centralizer. If $\vphi$ is a state, then one also has
    \[
        \|x\|_\vphi^2 = \sum_{\lambda\in \Lambda} \|\E_{(\vphi,\lambda)}(x)\|_\vphi^2,
    \]
for all $x\in M$ since $\Delta_\vphi = \sum_{\lambda\in \Lambda} 1_{\{\lambda\}}(\Delta_\vphi)$. See \cite{Dyk95, Dyk97} for additional details.

Suppose that $F$ is a finite-dimensional subfactor of a factor $N$, and $F^{c}$ is defined by $F^{c} = F' \cap N$.  Note that this means $N \cong F \otimes F^{c}$.  For a faithful normal state $\vphi$ on $N$, the following three conditions are equivalent (see, for example, \cite[Section 4]{Haa89}):
    \begin{enumerate}
    \item $F$ is $\Delta_\vphi$-invariant;
    \item there exists a $\vphi$-preserving conditional expectation from $N$ onto $F$;
    \item $(N, \vphi) \cong (F, \vphi) \otimes (F^{c}, \vphi)$.
    \end{enumerate}
For the purposes of this article, condition (3) above will be the one which is easiest to exploit. We will denote this more compactly as $\vphi=\vphi|_F\otimes \vphi|_{F^c}$.

\section{Extremal almost periodic weights}

In this section we present a few ways to construct new extremal almost periodic weights out of old ones. We will only need to apply them to \emph{states} in the present article, but we prove them in the generality of weights whenever possible because they may be of independent interest. The following lemma is proven in \cite[Lemma 5.2.(ii)]{HI22} in the case of states and for $M=R_\infty$ the Araki--Woods factor. However, the proof in the generality below is exactly the same.

\begin{lem}\label{lem:tensor_of_extremal_is_extremal}
Let $(M,\vphi)$ be a von Neumann algebra equipped with an extremal almost periodic weight. If $(N,\psi)$ is a factor equipped with an almost periodic weight satisfying $\Sd(\psi)\subset \Sd(\vphi)$, then $\varphi\otimes \psi$ is extremal almost periodic on $M\bar\otimes N$ with $\Sd(\vphi\otimes \psi) = \Sd(\vphi)$.
\end{lem}

The following can be viewed as a partial converse to the previous lemma.

\begin{lem}\label{lem:corners_of_ext_ap_weights_are_ext_ap}
Let $M=P\bar\otimes N$ be factor with $P$ a type $\mathrm{I}$ factor. Suppose $\vphi$ is an extremal almost periodic weight on $M$ satisfying: $\vphi|_P$ and $\vphi|_N$ are seminifinite; $\vphi|_P$ is almost periodic;  and $\vphi= \vphi|_P\otimes \vphi|_N$. Then $\vphi|_N$ is extremal almost periodic with $\Sd(\vphi|_N)=\Sd(\vphi)$. In particular, $\Sd(\vphi|_P)\subset \Sd(\vphi|_N)$.
\end{lem}
\begin{proof}
Identifying $P\cong B(\H)$ for some Hilbert space $\H$, we have that $\vphi|_P(x) = \Tr(h^{\frac12} x h^{\frac12})$ for some  positive (unbounded) operator $h$ on $\H$ and $\sigma_t^{\vphi|_P}(x) = h^{it} x h^{-it}$ for all $t\in \R$. Since $\vphi|_P$ is almost periodic, $h$ is necessarily diagonalizable with respect to some family of matrix units $\{e_{i,j}\in B(\H)\colon i,j\in I\}$, In particular one has $e_{i,i}\otimes 1\in M^\vphi$ for all $i\in I$. Because $(N,\vphi|_N)$ can be identified with $( (e_{i,i}\otimes 1) M (e_{i,i}\otimes 1), \frac{1}{\Tr(h e_{i,i})}\vphi)$ for any $i\in I$, it follows from \cite[Remark 1.5]{GGLN} that $\vphi|_N$ is extremal almost periodic.

Next, we see that $e_{i,j}$ is an eigenoperator of $\vphi|_P$ for each $i,j\in I$ since $h$ is diagonalized with respect to these matrix units, say with eigenvalue $\lambda_{i,j}$. Note that $\lambda_{j,i}=1/\lambda_{i,j}$.  Using that $M^\vphi$ is a factor we have $e_{i,i}M^\vphi e_{j,j} = e_{i,j}\otimes N$ is non-zero for each $i,j\in I$. Hence there is a non-zero $x\in N$ with $e_{i,j}\otimes x\in M^\vphi$, and the factorization $\vphi=\vphi|_P\otimes \vphi|_N$ implies $x\in N^{(\vphi|_N,1/\lambda_{i,j})} = N^{(\vphi|_N,\lambda_{j,i})}$. Consequently,
    \[
        \Sd(\vphi|_P) = \{\lambda_{i,j}\colon i,j\in I\} \subset \Sd(\vphi|_N).
    \]
Since $\Sd(\vphi|_N)$ is a group and $\Sd(\vphi)=\Sd(\vphi|_P)\Sd(\vphi|_N)$, it follows that $\Sd(\vphi|_N) = \Sd(\vphi)$ as claimed.
\end{proof}

Observe that if $\vphi$ in previous lemma is a state, then of course its restrictions to $P$ and $N$ are automatically semifinite, but also $\vphi|_P= \Tr(h^{\frac12}\,\cdot\, h^{\frac12})$ is automatically almost periodic since $h$ is necessarily trace class and hence diagonalizable.

\begin{lem}\label{lem:ext_ap_gives_ext_ap_for_nice_RD_deriv}
Let $M$ be a factor with separable predual equipped with an extremal almost periodic weight $\vphi$, let $\{p_i\in M^\vphi\colon i\in I\}$ be a family of projections satisfying $\sum_i p_i = 1$, and let $\{\mu_i\in \R_+\colon i\in I\}$ satisfy $\frac{\mu_i}{\mu_j}\in \Sd(\vphi)$ for all $i,j\in I$. Set
    \[
        h:= \sum_{i\in I} \mu_i p_i.
    \]
Then $\psi:=\vphi(h^\frac12\,\cdot\, h^\frac12)$ is an extremal almost periodic weight with $\Sd(\psi)=\Sd(\vphi)$.
\end{lem}
\begin{proof}
The weight $\psi$ is faithful since $h$ is non-singular, and since $\vphi$ is strictly semifinite and $p_i M^\vphi p_i \subset M^\psi$ it follows that $\psi$ is also strictly semifinite. Observe that $M^{(\psi,\eig)}$ contains $\text{span}\{ p_i M^{(\vphi,\eig)} p_j\colon i,j\in I\}$ and is therefore weak* dense. Consequently, $\psi$ is almost periodic by \cite[Lemma 1.4]{GGLN}. 

We next compute $\Sd(\psi)$. Denote $\Lambda:=\Sd(\vphi)$, which is a group by virtue of $\vphi$ being extremal. Fix $\lambda\in \Lambda$ and $i\in I$ and note that there exists a partial isometry $v\in M^{(\vphi,\lambda)}$. Since $v^*v$ and $p_i$ are not centrally orthogonal in $M^\vphi$, we can find a non-zero partial isometry $w\in M^\vphi$ satisfying $w^*w\leq p_i$ and $ww^*\leq v^*v$. Hence $v^*v wp_i=w$, and in particular $v w p_i$ is a non-zero partial isometry. Since $vwp_iw^*v^*$ and $p_i$ are also not centrally orthogonal in $M^\vphi$, we can find another partial isometry $u\in M^\vphi$ satisfying $u^*u \leq vwp_i w^* v^*$ and $uu^*\leq p_i$. Hence $p_i u vwp_i w^* v^* = u$, and in particular $p_i u vwp_i \neq 0$. Also, the Connes cocycle derivative theorem implies
    \[ 
        \sigma_t^\psi( p_i u vwp_i) = h^{it} \sigma_t^\vphi(p_i u vwp_i) h^{-it} = \mu_i^{it} p_iu \sigma_t^\vphi( v) w p_i \mu_i^{-it} = \lambda^{it} p_i u v w p_i.
    \]
Hence $\Lambda \subset \Sd(\psi)$. The reverse inclusion follows from $\text{span}\{ p_i M^{(\vphi,\eig)} p_j\colon i,j\in I\}$ being weakly dense: $\frac{\mu_i}{\mu_j}\in \Lambda$ implies $p_i M^{(\vphi,\lambda)} p_j \in \bigcup_{\nu \in \Lambda} M^{(\psi,\nu)}$.

Finally, we establish extremality. By replacing $\psi$ with $\frac{1}{\mu_{i_0}} \psi$ for some fixed $i_0\in I$, we may assume $\mu_i\in \Lambda$ for all $i\in I$. Let $G$ denote the compact dual to $\Lambda$ (which we give the discrete topology) and let $G\overset{\alpha}{\curvearrowright} M$ and $G\overset{\beta}{\curvearrowright} M$ be the point modular extensions of $\R \overset{\sigma^\vphi}{\curvearrowright} M$ and $\R \overset{\sigma^\psi}{\curvearrowright} M$, respectively. Consider the strongly continuous unitary representation
    \[
        G\ni s\mapsto u_s:= \sum_{i\in I} (s|\mu_i) p_i \in M^\vphi.
    \]
Then for $x\in M^{(\vphi,\lambda)}$ and $i,j\in I$ we have $\beta_s( p_i x p_j) = (s| \frac{\mu_i}{\mu_j} \lambda) p_i xp_j$ and
    \[
        u_s \alpha_s( p_i x p_j) u_s^* = (s|\mu_i) p_i \alpha_s(x) p_j (s|\mu_j)^{-1} = (s | \frac{\mu_i}{\mu_j} \lambda) p_i x p_j.
    \]
It follows that $\beta_s(x) = u_s \alpha_s (x) u_s^*$ for all $x\in M$ and hence $\alpha$ and $\beta$ are cocycle equivalent. Consequently, $M\rtimes_\alpha G\cong M\rtimes_\beta G$ by \cite[Theorem X.1.7.(ii)]{Tak03}. By \cite[Lemma 2.27]{GGLN}, the former is a factor since $M^\alpha = M^\vphi$ is a factor. Thus $M\rtimes_\beta G$ is a factor and we saw above that $\Sd(\psi)=\Lambda$, so the same lemma implies $M^\beta = M^\psi$ is a factor.
\end{proof}

The next lemma is a generalization of \cite[Lemma 5.4]{Haa89}) and will be needed in Section~\ref{sec:modular_cp}.

\begin{lem}\label{lem:state_preserving_embedding_of_hyperfinite_II1}
Let $M$ be a diffuse factor equipped with an extremal almost periodic state $\vphi$, and let $(R,\tau)$ be the unique separable injective type $\mathrm{II}_1$ factor equipped with its unique tracial state. For $h$ diagonlizable and affiliated with $R$, suppose $\psi=\tau(h^\frac12\,\cdot\,h^\frac12)$ is an almost periodic state on $R$ with $\Sd(\psi)\subset \Sd(\vphi)$. Then there exists a $\vphi$-invariant subfactor $P\leq M$ such that
    \[
        (R,\psi)\cong (P,\vphi|_P).
    \]
\end{lem}
\begin{proof}
Let
    \[
        h= \sum_{i\in I} \mu_i p_i
    \]
be the diagonalization of $h$, so that $p_i\in R$ for each $i\in I$ and $\Sd(\psi)=\{\frac{\mu_i}{\mu_j}\colon i,j\in I\}$. Since $M^\vphi$ is a  type $\mathrm{II}_1$ factor (see \cite[Lemma 3.1]{GGLN}), we can find a partition of unity $\{q_i \colon i\in I\}\subset M^\vphi$ satisfying $\vphi(q_i) = \psi(p_i) = \mu_i \tau(p_i)$ for all $i\in I$. Set
    \[
        k:= \sum_{i\in I} \mu_i q_i
    \]
and define $\omega:=\vphi(k^{-\frac12}\,\cdot\,k^{-\frac12})$. Then $\omega$ is an extremal almost periodic weight on $M$ by Lemma~\ref{lem:ext_ap_gives_ext_ap_for_nice_RD_deriv}. In fact, $\omega$ is a state because
    \[
        \omega(1) = \sum_{i\in I} \frac{1}{\mu_i} \vphi(q_i) = \sum_{i\in I} \tau(p_i) = 1.
    \]
Thus $M^\omega$ is a type $\mathrm{II}_1$ factor, and so we can find an embedding $\alpha\colon R\to M^\omega$ and set $P:=\alpha(M)$. Conjugating by a unitary in $M^\omega$ if necessary, we may assume $q_i\in P$ for each $i\in I$. We have $\omega\circ \alpha = \tau$ by uniqueness of the trace on $R$, so in particular for $f_i:=\alpha(p_i)$ we have
    \[
        \omega(f_i)= \tau(p_i) = \omega(q_i).
    \]
Thus $f_i$ is equivalent to $q_i$ in $P$, so let $u_i\in P$ be a partial isometry satisfying $u_i^* u_i=f_i$ and $u_iu_i^*=q_i$ and set $u:=\sum_{i\in I} u_i$. Then $u$ is a unitary satisfying $u f_i u^* = q_i$ for all $i\in I$, and hence $\beta:=\Ad{u}\circ \alpha$ is a unital $*$-isomorphism from $R$ to $P$ that further satisfies
    \[
        \beta(h)= \sum_{i\in I} \mu_i u\alpha(p_i)u^* = \sum_{i\in I} \mu_i q_i = k.
    \]
Consequently, for $x\in R$ we have
    \[
        \vphi(\beta(x))= \omega( k^{\frac12} \beta(x) k^{\frac12}) = \omega(\beta( h^{\frac12} x h^{\frac12})) = \omega(\alpha( h^{\frac12} x h^{\frac12})) = \tau(h^{\frac12} x h^{\frac12})=\psi(x),
    \]
where have used $u\in P\subset  M^\omega$ in the third equality. Therefore $(R,\psi)\cong (P,\vphi|_P)$. Finally, note that the Connes cocycle derivative theorem implies $\sigma_t^\vphi(x) = k^{it} \sigma_t^\omega(x) k^{-it}$, and since $k^{it}\in P$ for all $t\in \R$ it follows that $P$ is $\vphi$-invariant.
\end{proof}

Let $(M,\vphi)$ and $(R,\tau)$ be as in the pervious lemma. Suppose $\psi$ is a faithful state on $M_d(\C)$ with $\Sd(\psi)\subset \Sd(\vphi)$. Then $\psi= \tr_d(h^{\frac12}\,\cdot \, h^{\frac12})$ for some necessarily diagonalizable positive definite matrix $h\in M_d(\C)$, where $\tr_d$ is the normalized trace. Since $(M_d(\C)\otimes R, \tr_d\otimes \tau)\cong (R,\tau)$, we can apply the previous lemma to $\psi\otimes \tau$ to obtain the following almost periodic analogue of \cite[Proposition 4.6]{Haa89}. Note that $M_d(\C)\cong M_d(\C)\otimes \C$ is a $(\psi\otimes \tau)$-invariant subfactor of $M_d(\C)\otimes R$, and thus its image in $P$ will be $\vphi$-invariant.

\begin{lem}\label{lem:modular_embedding_matrix_algebra}
Let $M$ be a diffuse factor equipped with an extremal almost periodic state $\vphi$. For $d\in \N$ let $\psi$ be a faithful state on $M_d(\C)$ with $\Sd(\psi) \subset \Sd(\vphi)$. Then there exists a $\vphi$-invariant subfactor $F\leq M$ such that
    \[
        (M_d(\C),\psi)\cong (F, \vphi|_F).
    \]
\end{lem}

\section{Modular completely positive factorizations through matrix algebras}\label{sec:modular_cp}

Given a $(N,\vphi)$, a diffuse separable injective factor equipped with an extremal almost periodic state, we show in this section that the identity map on $N$ can be pointwise approximated in $\|\cdot\|_\vphi$-norm by a composition of completely maps to and from matrix algebras, which also respect the modular theory of $\vphi$. This result, Thoerem~\ref{thm:modular_finite_dimensional_approximations}, corresponds to \cite[Theorem 5.1]{Haa89} and will be needed in the proof of Theorem~\ref{thm:A}.

As an initial step, we first show such an approximations exist for completely positive maps to and from the unique separable injective type $\mathrm{II}_1$ factor $R$. The following lemma, which corresponds to \cite[Lemma 5.3]{Haa89}, is where the hypothesis of almost periodicity is most heavily used in the form of the point modular extension of the modular automorphism group of $\vphi$ and the existence of the idempotent maps $\E_{(\vphi,\lambda)}$ for $\lambda\in \Sd(\vphi)$.

\begin{lem}\label{lem:approximate_cp_embeddings_into_R}
Let $N$ be a diffuse separable injective factor, let $\vphi$ be an extremal almost periodic state on $N$, and let $(R,\tau)$ be the unique separable injective type $\rm{II}_1$ factor equipped with its unique tracial state. For every finite set $x_1,\ldots, x_n \in N$,  finite set $E\subset \Sd(\vphi)$, and $\epsilon>0$, there exist normal unital completely positive maps
    \[
        S\colon N\to R \qquad \text{ and } \qquad T\colon R\to N,
    \]
and a faithful state $\psi=\tau(a^\frac12\,\cdot\, a^\frac12)$ on $R$ with $E\subset \Sd(\psi)\subset \Sd(\vphi)$, where $a\in R_+$ has finite spectrum.
Moreover, one has
    \begin{align*}
        \psi\circ S&=\vphi,  & \vphi\circ T&=\psi,\\
        \sigma_t^{\psi}\circ S & = S\circ \sigma_t^\vphi, & \sigma_t^\vphi\circ T &= T\circ \sigma_t^{\psi},
    \end{align*}
for all $t\in \mathbb{R}$, and
    \[
        \| T\circ S(x_k) - x_k \|_\vphi <\epsilon \qquad k=1,\ldots, n.
    \]
\end{lem}
\begin{proof}
Enlarging $E$ if necessary, we may assume that
    \begin{align}\label{ineq:approx_by_eigenops}
        \left\| \sum_{\gamma\in E} \mathcal{E}_{(\vphi,\gamma)}(x_k) - x_k \right\|_\vphi < \frac{\epsilon}{2},
    \end{align}
for each $k=1,\ldots, n$. The multiplicative group generated by $E$ is finitely generated and torsion free, so it is free abelian and therefore has a basis. Let $\lambda_1,\ldots, \lambda_d\subset Sd(\vphi)$ be such such a basis. 
That is, $E\subset \{ \lambda_1^{r_1}\cdots \lambda_d^{r_d}\colon r_1,\ldots, r_d\in \Z\}$ and $\lambda_1^{r_1}\cdots \lambda_d^{r_d} = \lambda_1^{s_1}\cdots \lambda_d^{s_d}$ if and only if $r_1=s_1,\ldots, r_d=s_d$.

Now, for $p\in \N$ consider the density matrix
    \[
        h:= \frac{1}{\sum_{r_1,\ldots, r_d=1}^p \lambda_1^{r_1}\cdots \lambda_d^{r_d}} \sum_{r_1,\ldots, r_d=1}^p \lambda_1^{r_1}\cdots \lambda_d^{r_d} e_{r_1r_1}\otimes \cdots \otimes e_{r_d r_d} \in M_p(\C)^{\otimes d},
    \]
and let $\omega:= \Tr(h\,\cdot\,)$ be the associated state on $M_p(\C)^{\otimes d}$. Observe that
    \[
        \sigma_t^\omega( e_{r_1s_1}\otimes \cdots \otimes e_{r_d s_d}) = \lambda_1^{i(r_1 - s_1)t} \cdots \lambda_d^{i(r_d - s_d)t} e_{r_1s_1}\otimes \cdots \otimes e_{r_d s_d}.
    \]
Hence
    \[
        \Sd(\omega) = \{ \lambda_1^{r_1}\cdots \lambda_d^{r_d}\colon -p< r_1,\ldots, r_d < p\} \subset Sd(\vphi).
    \]
Lemma~\ref{lem:tensor_of_extremal_is_extremal} then implies that the centralizer of $N\otimes M_p(\C)^{\otimes d}$ with respect to $\chi:=\vphi\otimes \omega$ isomorphic to $R$ and $\Sd(\chi) = \Sd(\vphi)$. Equipping $\Sd(\vphi)$ with the discrete topology, let $G$ be its compact dual group with normalized Haar measure $\mu$, and let $G\overset{\alpha}{\curvearrowright} N$ and $G\overset{\beta}{\curvearrowright} N\otimes M_p(\C)^{\otimes d}$ be the point modular extensions $\R \overset{\sigma^\vphi}{\curvearrowright} N$ and $\R \overset{\sigma^\chi}{\curvearrowright} N\otimes M_p(\C)^{\otimes d}$, respectively. Then
    \[
        \beta_t(x\otimes e_{r_1s_1}\otimes \cdots \otimes e_{r_ds_d}) = (t\mid \lambda_1^{r_1-s_1}\cdots \lambda_d^{r_d-s_d}) \alpha_t(x)\otimes e_{r_1s_1}\otimes \cdots \otimes e_{r_ds_d},
    \]
so that
    \begin{align*}
        (N&\otimes M_p(\C)^{\otimes d})^\chi \\
        =& \left\{\sum_{r_1,s_1,\ldots, r_d,s_d=1}^p x_{r_1,s_1,\ldots r_d,s_d}\otimes e_{r_1s_1}\otimes\cdots \otimes e_{r_ds_d} \in N\otimes M_p(\C)^{\otimes d}  \colon  x_{r_1,s_1,\ldots r_d,s_d}\in N^{(\vphi,\lambda_1^{s_1-r_1}\cdots \lambda_d^{s_d-r_d})} \right\}
    \end{align*}
Also observe that for $\{x_{r_1,s_1,\ldots r_d,s_d}\in N\colon 1\leq r_1,s_1\ldots, r_d,s_d\leq p\}$ one has
    \begin{align*}
        \mathcal{E}_\chi &\left(\sum_{r_1,s_1,\ldots, r_d,s_d=1}^p x_{r_1,s_1,\ldots r_d,s_d}\otimes e_{r_1s_1}\otimes\cdots \otimes e_{r_ds_d} \right)\\
        &=\sum_{r_1,s_1,\ldots, r_d,s_d=1}^p \int_G ( \lambda_1^{r_1-s_1}\cdots \lambda_d^{r_d - s_d}\mid t)  \alpha_t(x_{r_1,s_1,\ldots r_d,s_d})\otimes e_{r_1s_1}\otimes \cdots \otimes e_{r_ds_d}\ d\mu(t)\\
        &=\sum_{r_1,s_1,\ldots, r_d,s_d=1}^p \int_G \overline{(\lambda_1^{s_1-r_1}\cdots \lambda_d^{s_r-r_d}\mid t)}  \alpha_t(x_{r_1,s_1,\ldots r_d,s_d})\otimes e_{r_1s_1}\otimes \cdots \otimes e_{r_ds_d}\ d\mu(t)\\
        &= \sum_{r_1,s_1,\ldots, r_d,s_d=1}^p \mathcal{E}_{(\vphi,\lambda_1^{s_1-r_1}\cdots \lambda_d^{s_d - r_d})}(x_{r_1,s_1,\ldots r_d,s_d}) \otimes e_{r_1s_1}\otimes \cdots \otimes e_{r_ds_d}.
    \end{align*}
Now, identifying $R\cong (N\otimes M_p(\C)^{\otimes d})^\chi$ we define $S\colon N\to R$ and $T\colon R\to N$ by
    \[
        S(x):= \sum_{r_1,s_1,\ldots, r_d,s_d=1}^p \mathcal{E}_{(\vphi,\lambda_1^{s_1-r_1}\cdots \lambda_d^{s_d - r_d})}(x) \otimes e_{r_1s_1}\otimes \cdots \otimes e_{r_ds_d}
    \]
for $x\in N$, and
    \[
        T\left( \sum_{r_1,s_1,\ldots, r_d,s_d=1}^p x_{r_1,s_1,\ldots r_d,s_d}\otimes e_{r_1s_1}\otimes\cdots \otimes e_{r_ds_d} \right) := \frac{1}{p^d} \sum_{r_1,s_1,\ldots, r_d,s_d=1}^p x_{r_1,s_1,\ldots r_d,s_d}
    \]
for $x_{r_1,s_1,\ldots r_d,s_d}\in N^{(\vphi,\lambda_1^{s_1-r_1}\cdots \lambda_d^{s_d-r_d})}$. Then $S$ and $T$ are unital by definition, and if we denote
    \[
        e_0:= \frac{1}{p^d} \sum_{r_1,s_1,\ldots, r_d,s_d=1}^p e_{r_1s_1}\otimes\cdots \otimes e_{r_ds_d},
    \]
then $S= p^d \mathcal{E}_{\chi}(\,\cdot\,\otimes e_0)$ and $T= (\text{id}_N\otimes \text{Tr})((1\otimes e_0) \,\cdot\, )$ are normal and completely positive. Moreover, for each $x\in N$ we have
    \[
        T\circ S(x) =  \frac{1}{p^d} \sum_{r_1,s_1,\ldots, r_d,s_d=1}^p \mathcal{E}_{(\vphi,\lambda_1^{s_1-r_1}\cdots \lambda_d^{s_d - r_d})}(x) = \sum_{|k_1|,\ldots, |k_d|<p} \left( 1- \frac{|k_1|}{p}\right)\cdots \left( 1 - \frac{|k_d|}{p}\right) \mathcal{E}_{(\vphi,\lambda_1^{k_1}\cdots \lambda_d^{k_d})}(x).
    \]
Suppose $p$ is large enough so that
    \[
        E\subset \{ \lambda_1^{k_1}\cdots \lambda_d^{k_d}\colon |k_1|,\ldots, |k_d|<p\}.
    \]
If we denote $y_k:= \sum_{\gamma\in E} \mathcal{E}_{(\vphi,\gamma)}(x_k)$ for each $k=1,\ldots, n$, then we have
    \begin{align*}
        \|T\circ S(x_k) - x_k\|_\vphi &\leq \|T\circ S(x_k - y_k)\|_\vphi + \|T\circ S(y_k) - y_k \|_\vphi + \|y_k - x_k\|_\vphi\\
        &< \epsilon + \left\| \sum_{\lambda_1^{k_1}\cdots \lambda_d^{k_d}\in E} \left[1- \left( 1- \frac{|k_1|}{p}\right)\cdots \left( 1 - \frac{|k_d|}{p}\right) \right] \mathcal{E}_{(\vphi,\lambda_1^{k_1}\cdots \lambda_d^{k_d})}(x_k) \right\|_\vphi,
    \end{align*}
where we have used (\ref{ineq:approx_by_eigenops}) and the fact that $\|T\circ S(x)\|_\vphi\leq \|x\|_\vphi$ (which follows from $T\circ S\colon N\to N$ being a unital completely positive map). Thus for sufficiently large $p$ one has 
    \[
        \|T\circ S(x_k) - x_k\|_\vphi < \epsilon
    \]
for each $k=1,\ldots, n$.

Finally, set $\psi:= \vphi\circ T$. Then for $x\in N$ one has
    \[
        \psi\circ S(x) = \vphi\circ T\circ S(x) =  \frac{1}{p^d} \sum_{r_1,s_1,\ldots, r_d,s_d=1}^p \vphi\circ\mathcal{E}_{(\vphi,\lambda_1^{s_1-r_1}\cdots \lambda_d^{s_d - r_d})}(x) = \frac{1}{p^d} \sum_{r_1,\ldots, r_d=1}^p \vphi(x) = \vphi(x).
    \]
Similarly, for $x_{r_1,s_1,\ldots r_d,s_d}\in N^{(\vphi,\lambda_1^{s_1-r_1}\cdots \lambda_d^{s_d-r_d})}$ one has
    \begin{align*}
        \psi&\left( \sum_{r_1,s_1,\ldots, r_d,s_d=1}^p x_{r_1,s_1,\ldots r_d,s_d}\otimes e_{r_1s_1}\otimes\cdots \otimes e_{r_ds_d}\right) \\
        &\qquad = \frac{1}{p^d} \sum_{r_1,s_1,\ldots, r_d,s_d=1}^p \vphi(x_{r_1,s_1,\ldots r_d,s_d}) = \frac{1}{p^d} \sum_{r_1,\ldots, r_d=1}^p \vphi(x_{r_1,r_1,\ldots r_d,r_d}).
    \end{align*}
Hence $\psi = \frac{1}{p^d}(\vphi\otimes \text{Tr})$. Note that this implies $\sigma_t^\psi = \sigma^{\vphi}_t\otimes \text{id}$, and thus by definition of $S$ and $T$ one has
    \begin{align*}
        \sigma_t^\psi\circ S = S\circ \sigma_t^\vphi \\
        \sigma_t^\vphi \circ T = T\circ \sigma_t^\psi.
    \end{align*}
Recalling that $\chi = \vphi\otimes \omega = (\vphi\otimes \text{Tr})( (1\otimes h)\,\cdot\,)$, we see that if
    \[
        a: = \frac{1}{p^d}(1\otimes h^{-1}),
    \]
then $\psi = \chi(a\,\cdot\,)$. Since $(R,\tau)\cong ( (N\otimes M_p(\C)^{\otimes d})^\chi, \chi)$, this implies $\Sd(\psi)$ is the finite set consisting of the ratios of eigenvalues of $h$ and so $E\subset \Sd(\psi)\subset \Sd(\vphi)$.
\end{proof}

We require some additional notation for the proof of the main theorem in this section. For a finite subset $E\subset \R_+$ we denote
    \begin{align*}
        (\mathbb{M}_E, \phi_E)&:=\bigotimes_{\lambda\in E} (M_2(\C),\phi_\lambda)\\
         (R_E,\varphi_E)&:=\overline{\bigotimes_{\lambda\in E}} (R_\lambda,\vphi_\lambda).
    \end{align*}
Note that $\vphi_E$ is extremal almost periodic and
    \begin{align}\label{eqn:Sd_invariant_of_matrix_state}
        \Sd(\phi_E) = \prod_{\lambda\in E} \Sd(\phi_\lambda) = \left\{ \prod_{\lambda\in E} \lambda^{e(\lambda)}\colon e(\lambda)\in \{-1,0,1\}\ \forall \lambda\in E \right\} \leq \<E\> = \Sd(\vphi_E).
    \end{align}
In particular, $\Sd(\phi_E)$ is still a finite set. For any $m\in \N$ we have a state-preserving embedding
    \[
       (\mathbb{M}_E^{\otimes m}, \phi_E^{\otimes m}) \hookrightarrow (R_E, \vphi_E),
    \]
given by mapping the $m$ tensor factors of $(M_2(\C),\phi_\lambda)$ into the first $m$ tensor factors of 
    \[
        (R_\lambda,\vphi_\lambda) = \bigotimes_{n\in \N} (M_2(\C),\phi_\lambda).
    \]
It is easy to see that this image is $\vphi_E$-invariant, and so there exists a $\vphi_E$-preserving faithful normal conditional expectation $\E_{E,m}\colon R_E \to \mathbb{M}_E^{\otimes m}$.

\begin{thm}\label{thm:modular_finite_dimensional_approximations}
Let $N$ be a diffuse separable injective factor equipped with an extremal almost periodic state $\vphi$. For every finite set $x_1,\ldots, x_n\in N$, finite set $E_0\subset \Sd(\vphi)$,  and $\epsilon>0$, there exists a finite subset $E_0\subset E\subset \Sd(\vphi)$, $m\in \N$, and normal unital completely positive maps
    \[
        S\colon N\to \mathbb{M}_E^{\otimes m} \qquad \text{ and } \qquad T\colon \mathbb{M}_E^{\otimes m} \to N,
    \]
such that
    \begin{align*}
        \phi_E^{\otimes m}\circ S&= \vphi, & \vphi\circ T&= \phi_E^{\otimes m} \\
        \sigma_t^{\phi_E^{\otimes m}}\circ S & = S\circ \sigma_t^{\vphi}, & \sigma_t^\vphi \circ T &= T\circ \sigma_t^{\phi_E^{\otimes m}}
    \end{align*}
for all $t\in \R$, and
    \[
        \| T\circ S(x_k) - x_k\|_\vphi < \epsilon \qquad k=1,\ldots, n
    \]
\end{thm}
\begin{proof}
Apply Lemma~\ref{lem:approximate_cp_embeddings_into_R} to $x_1,\ldots, x_n$, $\frac{\epsilon}{2}$, and  $E_0$ to find normal unital completely positive maps $S'\colon N\to R$ and $T'\colon R\to N$ and a faithful normal state $\psi$ on $R$ with their suite of properties. Let $E:=\Sd(\psi)$, which satisfies $E_0 \subset E \subset \Sd(\vphi)$. Applying Lemma~\ref{lem:state_preserving_embedding_of_hyperfinite_II1} to $(R_E,\vphi_E)$ and $(R,\psi)$ (note that (\ref{eqn:Sd_invariant_of_matrix_state}) gives $E\subset \Sd(\vphi_E)$), we can identify $R$ as a $\vphi_E$-invariant subalgebra of $R_E$ satisfying $\vphi_E|_R = \psi$. Consequently, there exists a faithful normal $\vphi_E$-preserving conditional expectation $\E_R: R_E\to R$, and so we define normal unital completely positive maps
    \[
        S''\colon N\to R_E \qquad \text{ and } \qquad T''\colon R_E \to N,
    \]
by letting $S''$ be $S'$ but viewed as valued in $R_E$ and letting $T'':=T'\circ \E_R$. It follows that
    \begin{align*}
        \vphi_E\circ S'' &= \psi\circ S' = \vphi\\
        \vphi\circ T'' & = \vphi\circ T'\circ  \E_R = \psi\circ \E_R = \vphi_E
    \end{align*}
and
    \begin{align*}
        \sigma_t^{\vphi_E}\circ S'' &= \sigma_t^{\psi}\circ S' = S'\circ \sigma_t^{\vphi} = S''\circ \sigma_t^{\vphi}\\
        \sigma_t^{\vphi}\circ T'' &= \sigma_t^{\vphi} \circ T'\circ \E_R = T'\circ \sigma_t^{\psi}\circ \E_R = T''\circ \sigma_t^{\vphi_E},
    \end{align*}
for all $t\in \R$. Moreover, for each $k=1,\ldots, n$ we have
    \[
        \| T''\circ S''(x_k)) - x_k\|_\vphi = \| T'\circ \E_R \circ S'(x_k) -x_k\|_\vphi = \| T' \circ S'(x_k) -x_k\|_\vphi < \frac{\epsilon}{2}.
    \]
Now, as in the discussion preceding the statement of the theorem we view $\mathbb{M}_E^{\otimes m}$ as a subfactor of $R_E$ with $\vphi_E$-preserving conditional expectation $\E_{E,m}\colon R_E \to \mathbb{M}_E^{\otimes m}$ for each $m\in \N$. Then for $x\in R_E $ we have
    \[
        \lim_{m\to\infty} \| \E_{E,m}(x) - x \|_{\vphi_E}=0.
    \]
Choose $m\in \N$ large enough so that
    \[
        \| \E_{E,m}\circ S''(x_k) - S''(x_k)\|_{\vphi_E} < \frac{\epsilon}{2},
    \]
for each $k=1,\ldots, n$. Finally, we define the desired normal unital completely positive maps $S$ and $T$ by
    \[
        S:=\E_{E,m}\circ S'' \qquad \text{ and } \qquad T:=T''|_{\mathbb{M}_E^{\otimes m}}.
    \]
The claimed relations with $\vphi$ and $\sigma_t^\vphi$, $t\in \R$, are immediate, and for each $k=1,\ldots, n$ we have
    \begin{align*}
        \| T\circ S(x_k) - x_k \|_\vphi &\leq \| (T\circ S - T''\circ S'')(x_k)\|_\vphi + \|T''\circ S''(x_k) - x_k \|_\vphi\\
        &= \| T'' \left[\E_{E,m}\circ S''(x_k) - S''(x_k)\right] \|_\vphi + \|T''\circ S''(x_k) - x_k \|_\vphi < \frac{\epsilon}{2} + \frac{\epsilon}{2} = \epsilon,
    \end{align*}
as desired.
\end{proof}

\section{Approximate Kraus representations}\label{sec:approximate_Kraus_representations}

For $(N,\vphi)$ a diffuse factor equipped with an extremal almost periodic state, the results in this section concern $\vphi$-invariant finite dimensional subfactors $F\leq N$ and normal unital completely positive maps $T\colon F\to N$ that are compatible with $\vphi$ and its modular theory. We note that Theorem~\ref{thm:modular_finite_dimensional_approximations} yields such maps after using Lemma~\ref{lem:modular_embedding_matrix_algebra} to identify $\mathbb{M}_E^{\otimes}$ with such a subfactor $F\leq N$.

Our first lemma, corresponds to \cite[Lemma 4.4]{Haa89}, which  shows such maps $T$ admit Kraus representations that use only elements of the centralizer $N^\vphi$.

\begin{lem}\label{lem:centralizer_Kraus_rep}
Let $N$ be a factor equipped with an extremal almost periodic state $\vphi$, and let $F\leq N$ be a finite dimensional subfactor for which
    \[
        \vphi= \vphi|_{F} \otimes \vphi|_{F^c}.
    \]
If $T\colon F\to N$ is a normal completely positive map satisfying $\sigma_t^\vphi\circ T = T\circ \sigma_t^{\vphi|_F}$, then there exists a finite set $a_1,\ldots, a_n\in N^\vphi$ such that
    \[
        T(x) = \sum_{j=1}^n a_j^* x a_j
    \]
for all $x\in F$.
\end{lem}
\begin{proof}
We can find a family of matrix units $\{e_{i,j}\in F\colon 1\leq i,j\leq d\}$ so that $\vphi|_F= \Tr(h\,\cdot\,)$ and
    \[
        h=\sum_{i=1}^d \mu_i e_{i,i},
    \]
where $\sum_i \mu_i=1$. Then $\{\frac{\mu_i}{\mu_j}\colon 1\leq i,j\leq d\}=\Sd(\vphi|_F) \subset \Sd(\vphi)$. The complete positivity of $T$ implies
    \[
        a:= \sum_{i,j=1}^d T(e_{i,j}) \otimes e_{i,j}
    \]
is a  positive operator in $N\otimes F$, so we can define $b:=a^{1/2}$. Then $b$ is of the form
    \[
        b= \sum_{i,j=1}^d b_{i,j}\otimes e_{i,j}
    \]
for  $b_{i,j}\in N$ and
    \[
        T(e_{i,j}) = (1\otimes \Tr)( a(1\otimes e_{j,i})) = (1\otimes \Tr)( b^*b(1\otimes e_{j,i})) = \sum_{k=1}^n b_{k,i}^*b_{k,j}.
    \]
If we set $c_{i,j}:= \sum_{k=1}^d e_{k,j}b_{i,k}$ for each $1\leq i,j\leq d$, then it follows that
    \[
        \sum_{i,j=1}^d c_{i,j}^* e_{k,\ell} c_{i,j} = \sum_{i,j,r,s=1}^d b_{i,r}^* e_{j,r} e_{k,\ell} e_{s,j}b_{i,s}  = \sum_{i,j=1}^d b_{i,r}^* e_{j,j} b_{i,s} = T(e_{k,\ell}).
    \]
Therefore
    \[
        T(x) = \sum_{i,j=1}^d c_{i,j}^* x c_{i,j}
    \]
for all $x\in F$. Next we will replace the $c_{i,j}$ with elements of the centralizer. Define a new state on $F$ by $\psi:= \Tr(h^{-1}\,\cdot\,)$. Then using $\sigma_t^\vphi\circ T = T\circ \sigma_t^{\vphi|_F}$ we have
    \[
        \sigma_t^{\vphi\otimes \psi}(a) = \sum_{i,j=1}^d \sigma_t^{\vphi}(T(e_{i,j})) \otimes \sigma_t^{\psi}(e_{i,j}) = \sum_{i,j=1}^d (\frac{\mu_i}{\mu_j})T(e_{i,j})\otimes \frac{\mu_j}{\mu_i}e_{i,j} = a.
    \]
Hence $a\in (N\otimes F)^{\vphi\otimes \psi}$ and therefore $b$ lies in this centralizer as well. Running a similar computation as above on the expanded formula for $b$ yields $\sigma_t^\vphi(b_{i,j}) = \frac{\mu_i}{\mu_j}b_{i,j}$ for all $i,j=1,\ldots, d$, and therefore
    \[
        \sigma_t^\vphi(c_{i,j}) = \sum_{k=1}^d \frac{\mu_k}{\mu_j}e_{k,j} \frac{\mu_i}{\mu_k} b_{i,k} = \frac{\mu_i}{\mu_j} c_{i,j}
    \]
for all $i,j=1,\ldots, d$.

Next, observe that $\vphi|_{F^c}$ is extremal almost periodic with $ \text{Sd}(\vphi|_F) \subset \text{Sd}(\vphi|_{F^c})$ by Lemma~\ref{lem:corners_of_ext_ap_weights_are_ext_ap}. We claim that for each $\frac{\mu_i}{\mu_j} \in \Sd(\vphi|_F)$ there exists a finite set $\mathcal{V}_{i,j}\subset (F^c)^{(\vphi|_{F^c},\frac{\mu_i}{\mu_j})}$ satisfying
    \[
        \sum_{v\in \mathcal{V}_{i,j}} vv^* = 1.
    \]
Indeed, we first use \cite[Lemma 2.1]{GGLN} to a find a (potentially infinite) family $\mathcal{W}\subset (F^c)^{(\vphi|_{F^c},\frac{\mu_i}{\mu_j})}$ satisfying
    \[
        \sum_{v\in \mathcal{W}} vv^* =1.
    \]
For $\frac{\mu_i}{\mu_j}\geq 1$, this same lemma implies we can chose $\mathcal{V}_{i,j}=\mathcal{W}$ of size one. Thus it suffices to consider the case $\frac{\mu_i}{\mu_j}<1$. For some $m\in\N$ we can partition $\mathcal{W}= \mathcal{W}_1\sqcup \cdots \sqcup \mathcal{W}_m$ so that
    \[
        \sum_{v\in \mathcal{W}_k} \vphi(vv^*) \leq \frac{\mu_i}{\mu_j}
    \]
for each $k=1,\ldots, m$. Consequently,
    \[
        \sum_{v\in \mathcal{W}_k} \vphi(v^*v) = \sum_{v\in \mathcal{W}_k} \frac{\mu_j}{\mu_i} \vphi(vv^*) \leq 1,
    \]
and so by precomposing the $v\in \mathcal{W}_k$ with suitable partial isometries  in $(F^c)^{\vphi|_{F^c}}$ we can assume the range and source projections are pairwise orthogonal over $\mathcal{W}_k$ (this uses the factoriality of $(F^c)^{\vphi|_{F^c}}$). Setting $v_k:= \sum_{v\in \mathcal{W}_k} v$ for each $k=1,\ldots, m$, the desired finite set is then $\mathcal{V}_{i,j}:=\{v_1,\ldots, v_m\}$.

Now, $v^*c_{i,j}\in N^\vphi$ for each $1\leq i,j\leq d$ and $v\in \mathcal{V}_{i,j}$, and since $F^c$ commutes with $F$ inside $N$ we have
    \[
        \sum_{i,j=1}^d \sum_{v\in \mathcal{V}_{i,j}} (v^* c_{i,j})^* x (v^*c_{i,j}) = \sum_{i,j=1}^d \sum_{v\in \mathcal{V}_{i,j}} c_{i,j}^* x vv^* c_{i,j} = \sum_{i,j=1}^d c_{i,j}^* x c_{i,j} = T(x)
    \]
for each $x\in F$. Thus $\{ v^*c_{i,j}\colon 1\leq i, j\leq d,\ v\in \mathcal{V}_{i,j}\}$ is the desired finite set in $N^\vphi$.
\end{proof}

The next pair of lemmas are relative Dixmier norm approximation results. These relate to \cite[Corollary 3.4 and Lemma 4.5]{Haa89}, respectively, but are necessarily weaker in the case that $\Sd(\vphi)$ is dense in $\R_+$ (see the proof of \cite[Theorem 3.1]{Haa89}).

\begin{lem}\label{lem:norm_Dixmier_on_eigenops}
Let $N$ be a diffuse factor with an extremal almost periodic state $\vphi$. For any finite set $x_1,\ldots, x_n\in N^{(\vphi,\eig)}$ and $\epsilon > 0$ there exists $\alpha\in \conv\{\Ad{u}\colon u\in N^\vphi\}$ satisfying
    \[
        \|\alpha(x_k) - \vphi(x_k)\| <\epsilon
    \]
for each $k=1,\ldots, n$.
\end{lem}
\begin{proof}
We first observe that for $\lambda\in \R_+$ and $x\in N^{(\vphi,\lambda)}$ one has
    \[
        \vphi(x) \in \overline{\conv\{ u^{*}xu\, | \, u \in U(N_{\vphi)}\}}^{\|\cdot\|}.
    \]
Indeed, for $\lambda=1$ this is simply the Dixmier approximation theorem for type $\mathrm{II}_1$ factors (see \cite[Section III.5.1]{Dix81}), and for $\lambda\neq 1$ this follows from the same argument as in the proof of \cite[Lemma 3.2]{Haa89}. (Note that $N^\vphi$ is indeed a type $\mathrm{II}_1$ factor by \cite[Lemma 3.1]{GGLN}.)

Now, let $E\subset \Sd(\vphi)$ be the necessarily finite set such that $x_1,\ldots, x_n \in \text{span}\{N^{(\vphi,\lambda)}\colon \lambda\in E\}$, and for each $k=1,\ldots,n$ write $x_k = \sum_{\lambda\in E} x_{k,\lambda}$ where $x_{k,\lambda}\in N^{(\vphi,\lambda)}\cup\{0\}$. Iteratively applying our initial observation through some enumeration of $\{1,\ldots, n\}\times E$ yields $\alpha\in \conv\{\Ad{u}\colon u\in N^\vphi\}$  satisfying
    \[
        \|\alpha(x_{k,\lambda}) - \vphi(x_{k,\lambda}) \| < \frac{\epsilon}{|E|}
    \]
for each $k=1,\ldots, n$ and $\lambda \in E$, and so the claimed inequality for $x_k$ follows.
\end{proof}

\begin{lem}\label{lem:norm_Dixmier_approximation_of_expectation}
Let $N$ be a diffuse factor equipped with an extremal almost periodic state $\vphi$, and let $F\leq N$ be a finite dimensional subfactor for which
    \[
        \vphi= \vphi|_{F} \otimes \vphi|_{F^c}.
    \]
If $\E_F\colon N\to F$ is the $\vphi$-invariant conditional expectation, then for all $x\in N^{(\vphi,\eig)}$ and $\epsilon>0$ there exists $\alpha\in \conv\{\Ad{u}\colon u\in \mathcal{U}(F^c\cap N^\vphi)\}$ satisfying
    \[
        \| \alpha(x) - \E_F(x) \| <\epsilon.
    \]
\end{lem}
\begin{proof}
For $x\in N^{(\vphi,\eig)}$ write
    \[
        x= \sum_{\lambda\in E} x_\lambda,
    \]
where $x_\lambda\in N^{(\vphi,\lambda)}$ for each $\lambda$ in a finite subset $E\subset \Sd(\vphi)$. Fix a family of matrix units $\{e_{i,j}\colon 1\leq i,j\leq d\}$ for $F$  consisting of eigenoperators for $\sigma^{\vphi|_F}$. In fact, $\vphi=\vphi|_F\otimes \vphi|_{F^c}$ implies they will also be eigenoperators for $\sigma^\vphi$, say $\sigma_t^\vphi(e_{i,j}) = \lambda_{i,j}^{it} e_{i,j}$ for $\lambda_{i,j}\in \Sd(\vphi)$. For each $\lambda\in E$ write
    \[
        x_\lambda = \sum_{i,j=1}^d e_{i,j} x_{i,j}^{(\lambda)},
    \]
where $x_{i,j}^{(\lambda)}\in F^c$ for all $1\leq i,j\leq d$. Our choice of matrix units and $x_\lambda\in N^{(\vphi,\lambda)}$ implies $x_{i,j}^{(\lambda)}\in N^{(\vphi, \frac{\lambda}{\lambda_{i,j}})}$ for each $1\leq i ,j\leq d$. Also, the tensor decomposition of $\vphi$ implies
    \[
        \E_F(x_\lambda) = \sum_{i,j=1}^d e_{i,j} \vphi(x_{i,j}^{(\lambda)}).
    \]
Lemma~\ref{lem:corners_of_ext_ap_weights_are_ext_ap} tells us that $\vphi|_{F^c}$ is extremal almost periodic, and so we can apply Lemma~\ref{lem:norm_Dixmier_on_eigenops} to $\{x_{i,j}^{(\lambda)}\colon \lambda\in E,\ 1\leq i,j\leq d\} \subset F^c$ to find $\alpha\in \conv\{\Ad{u}\colon u\in \mathcal{U}(F^c\cap N^\vphi)\}$ satisfying
    \[
        \|\alpha(x_{i,j}^{(\lambda)}) - \varphi(x_{i,j}^{(\lambda)})\| < \frac{\epsilon}{d^2|E|}.
    \]
(Note that $F^c\cap N^\vphi = (F^c)^{\vphi|_{F^c}}$.) Since $F^c$ commutes with $F$ we have
    \[
        \| \alpha(x) - \E_F(x)\| = \left\| \sum_{\lambda\in E} \sum_{i,j=1}^d e_{i,j}\left[\alpha(x_{i,j}^{(\lambda)}) - \vphi(x_{i,j}^{(\lambda)})\right] \right\| < \epsilon,
    \]
as claimed.
\end{proof}

The following theorem is our main result in this section and corresponds to \cite[Theorem 4.1]{Haa89}. Recall that Lemma~\ref{lem:centralizer_Kraus_rep} gave an exact Kraus representation of the completely positive map $T(x) = \sum_{j=1}^n a_j^* x a_j $ for $a_1,\ldots, a_n \in N^\vphi$. Since $T$ is unital, we in particular have that $\sum_{j=1}^n a_j^* a_j=1$. However, $\sum_{j=1}^n a_j a_j^*=1$ may fail to hold, and it will be needed in the proof of Theorem~\ref{thm:A} to apply \cite[Theorem 2.3]{Haa89}. In the following theorem, we obtain this latter relation by sacrificing our exact finite Kraus representation for an infinite approximate Kraus representation.

\begin{thm}\label{thm:unital_but_infinite_Kraus_approximation}
Let $N$ be a diffuse factor equipped with an extremal almost periodic state $\vphi$, and let $F\leq N$ be a finite dimensional subfactor for which
    \[
        \vphi= \vphi|_{F} \otimes \vphi|_{F^c}
    \]
Suppose $T\colon F\to N$ is a normal unital completely positive map satisfying $\vphi\circ T= \vphi|_F$ and $\sigma_t^{\vphi}\circ T = T \circ \sigma_t^{\vphi|_F}$. Then for every $\delta>0$ there exists a sequence $(a_n)_{n\in \N}\subset N^\vphi$ such that
    \[
        \sum_{n=1}^\infty a_n^* a_n = \sum_{n=1}^\infty a_n a_n^*=1
    \]
and
    \[
        \left\| T(x) - \sum_{n=1}^\infty a_n^* x a_n  \right\| \leq \delta \|x\|
    \]
for all $x\in F$.
\end{thm}
\begin{proof}
We first apply Lemma~\ref{lem:centralizer_Kraus_rep} to obtain $b_1,\ldots, b_n\in N^\vphi$ so that
    \[
        T(x) = \sum_{i=1}^n b_i^* x b_i,
    \]
for all $x\in F$. In particular, this gives
    \[
        \sum_{i=1}^n b_i^* b_i = T(1) = 1.
    \]
Let $\E_F\colon N\to F$ to be $\sigma^\vphi$-preserving conditional expectation. Using $\vphi\circ T = \vphi|_F$ and $b_i\in N^\vphi$, we have for $x\in F$
    \[
        \< \E_F\left(\sum_{i=1}^d b_i b_i^* \right), x\>_\vphi = \< \sum_{i=1}^d b_i b_i^*, x\>_\vphi = \< 1, \sum_{i=1}^d b_i^* x b_i\>_\vphi = \<1, T(x)\>_\vphi = \<1, x\>_\vphi.
    \]
Hence $\E_F(\sum_{i=1}^d b_i b_i^*) = 1$.

Now, apply Lemma~\ref{lem:norm_Dixmier_approximation_of_expectation} to $\sum_{i=1}^d b_i b_i^*$ to find $\alpha = \sum_{j=1}^{d} t_j \Ad{u_j}\in \conv\{ \Ad{u}\colon u\in \mathcal{U}(F^c\cap N^\vphi)\}$ satisfying
    \[
        \left\| \alpha\left(\sum_{i=1}^n b_ib_i^* \right) -1 \right\| < \frac{\delta}{2}.
    \]
Set $c_{i,j}:=t_j^{\frac12}u_j b_i$. Then the above inequality gives
    \begin{align}\label{eqn:Kraus_operator_norm_bound}
        \| \sum_{i=1}^n \sum_{j=1}^d c_{i,j} c_{i,j}^* - 1 \| < \frac{\delta}{2}.
    \end{align}
We also have
    \[
        \sum_{i=1}^n \sum_{j=1}^d  c_{i,j}^*c_{i,j} = \sum_{j=1}^d t_j \sum_{i=1}^n b_i^* b_i = 1,
    \]
and for $x\in F$ we have
    \[
        \sum_{i=1}^n \sum_{j=1}^d c_{i,j}^* x c_{i,j} = \sum_{j=1}^d t_j \sum_{i=1}^n b_i^* x b_i = T(x).
    \]
Reindex the operators $c_{i,j}$ as $c_1,\ldots, c_m$ for some $m\in \N$, and then set
    \[
        a_i:= \left( 1- \frac{\delta}{2} \right)^{\frac12} c_i.
    \]
for each $i=1,\ldots, m$. Then
    \[
        \sum_{i=1}^m a_i^* a_i = 1 - \frac{\delta}{2} < 1,
    \]
and (\ref{eqn:Kraus_operator_norm_bound}) implies
    \[
        \sum_{i=1}^m a_i a_i^*  = \left[\sum_{i=1}^m a_i a_i^* - \left( 1-\frac{\delta}{2}\right) \right] + \left( 1-\frac{\delta}{2}\right)  < \left( 1-\frac{\delta}{2}\right) \left( \frac{\delta}{2}\left( 1-\frac{\delta}{2}\right)   + 1\right) < 1.
    \]
Recalling that $N^\vphi$ is a type $\mathrm{II}_1$ factor by \cite[Lemma 3.1]{GGLN}, we can use \cite[Lemma 5.1]{Haa85} to find $\{ a_i\in N^\vphi\colon i\geq m+1\}$ so that
    \[
        \sum_{i=1}^\infty a_i^* a_i = \sum_{i=1}^\infty a_i a_i^* =1.
    \]
Note that $\sum_{i\geq m+1} a_i^*a_i = \frac{\delta}{2}$, and consequently
    \[
        S(x):= \sum_{i=m+1}^\infty a_i^* x a_i
    \]
is a completely positive map with $\|S\| = \| S(1)\| = \frac{\delta}{2}$. Thus $\|S(x)\|\leq \frac{\delta}{2}\|x\|$, and recall that $\|T(x)\|\leq \|x\|$ since $\|T\|=\|T(1)\| = 1$. Hence
    \[
        \left\| T(x) - \sum_{i=1}^\infty a_i^* x a_i \right\| = \left\| T(x) - \left( 1-\frac{\delta}{2}\right) T(x) - S(x) \right\| = \left\| \frac{\delta}{2}T(x) - S(x)\right\| \leq \delta\|x\|
    \]
for all $x\in F$.
\end{proof}

\section{Proof of Theorem~\ref{thm:A}}\label{sec:proof_thmA}

Let $N$ be a separable injective factor equipped with an extremal almost periodic state $\vphi$. For $\Lambda:=\Sd(\vphi)$, let $(R_\Lambda, \vphi_\Lambda)$ be as in (\ref{eqn:def_of_R_Lambda}). Towards constructing the isomorphism $(N,\vphi)\cong (R_\Lambda, \vphi_\Lambda)$, we first prove three claims corresponding to \cite[Lemmas 6.2, 6.3, and 6.4]{Haa89}, respectively.\\

{\leftskip 15 pt
\noindent\textbf{Claim 1:} For each finite set $u_1,\ldots, u_n\in \mathcal{U}(N)$, finite set $E_0\subset \Sd(\vphi)$, and $\delta>0$ there exists a finite set $E_0\subset E\subset \Sd(\vphi)$, $m\in \N$, a completely positive map $T\colon \mathbb{M}_E^{\otimes m}\to N$, and unitarty operators $v_1,\ldots, v_n\in \mathbb{M}_E^{\otimes m}$ such that
    \begin{align*}
        \vphi\circ T &= \phi_E^{\otimes m}\\
        \sigma_t^\vphi\circ T &= T \circ \sigma_t^{\phi_E^{\otimes m}}
    \end{align*}
for all $t\in \R$, and
    \begin{align*}
        \|T(v_k) - u_k\|_\vphi < \delta
    \end{align*}
for each $k=1,\ldots, n$.

}

\hfill

\noindent Choose $\epsilon>0$ satisfying $(2\epsilon)^{\frac12} + \epsilon <\delta$. Applying Theorem~\ref{thm:modular_finite_dimensional_approximations} to $u_1,\ldots, u_n\in N$, $E_0$, and this $\epsilon$ we obtain a finite set $E_0\subset E\subset \Sd(\vphi)$, $m\in \N$, and normal unital completely positive maps $S\colon N\to \mathbb{M}_E^{\otimes m}$ and $T\colon \mathbb{M}_E^{\otimes m}\to N$ with their collection of properties. Set $y_k:=S(u_k)$ so that $\|y_k\|\leq 1$ and
    \[
        \|T(y_k) - u_k \|_\vphi < \epsilon
    \]
for each $k=1,\ldots, n$. Observe that
    \[
        \|y_k\|_{\phi_E^{\otimes m}} \geq \|T(y_k)\|_\vphi \geq \|u_k\|_\vphi - \| T(y_k) - u_k\|_\vphi > 1-\epsilon.
    \]
Let $y_k = v_k |y_k|$ be the polar decomposition, and since $\mathbb{M}_E^{\otimes m}$ is a finite factor we can extend $v_k$ to a unitary in this algebra. Thus
    \[
        \| v_k - y_k\|_{\phi_E^{\otimes m}}^2 = \| 1- |y_k|\|_{\phi_E^{\otimes m}}^2 \leq 1- \| |y_k| \|_{\phi_E^{\otimes m}}^2 = 1 - \| y_k\|_{\phi_E^{\otimes m}}^2 < 1- (1-\epsilon)^2 < 2\epsilon,
    \]
where in the first inequality we have used $|y_k|^2 + (1 - |y_k|)^2 \leq 1$, which follows from $0\leq |y_k|\leq \| y_k\|\leq 1.$ Consequently we have
    \[
        \| T(v_k) - u_k \|_\vphi \leq \|T(v_k - y_k)\|_\vphi + \|T(y_k) - u_k\|_\vphi < (2\epsilon)^{\frac12} + \epsilon < \delta,
    \]
by our original assumption on $\epsilon$. This proves Claim 1.\\

{\leftskip 15pt
\noindent\textbf{Claim 2:} For each finite set $u_1,\ldots, u_n\in \mathcal{U}(N)$, finite set $E_0\subset \Sd(\vphi)$, and $\delta>0$ there exists a finite dimensional subfactor $F\leq N$, unitary operators $v_1,\ldots, v_n\in \mathcal{U}(F)$, and a sequence $(a_i)_{i\in \N}\subset N^\vphi$ satisfying
    \begin{enumerate}[leftmargin=5em, label=(\roman*)]
        \item $\vphi = \vphi|_F\otimes \vphi|_{F^c}$;

        \item $(F,\vphi|_F)\cong (\mathbb{M}_E^{\otimes m}, \phi_E^{\otimes m})$ for some finite set $E_0\subset E\subset \Sd(\vphi)$ and $m\in \N$;
        
        \item $\sum_{i=1}^\infty a_i^* a_i = \sum_{i=1}^\infty a_ia_i^*=1$;
        
        \item $\sum_{i=1}^\infty \| v_k a_i - a_i u_k\|_\vphi^2 < \delta$ for each $k=1,\ldots, n$.
    \end{enumerate}

}

\hfill

\noindent First, we apply Claim 1 to $u_1,\ldots, u_n$, $E_0$, and $\frac{\delta}{4}$ and let $E_0\subset E\subset \Sd(\vphi)$, $m\in \N$, $T\colon \mathbb{M}_E^{\otimes m}\to N$, and $v_1,\ldots, v_n\in \mathbb{M}_E^{\otimes m}$ the resulting data. Then applying Lemma~\ref{lem:modular_embedding_matrix_algebra} to $(\mathbb{M}_E^{\otimes m}, \phi_E^{\otimes m})$ gives a finite dimensional subfactor $F\leq N$ satisfying (i) and (ii).

Observe that identifying $(\mathbb{M}_E^{\otimes m}, \phi_E^{\otimes m})$ with $(F,\vphi|_F)$ gives $\vphi\circ T = \vphi|_F$ and $\sigma_t^\vphi\circ T = T\circ \sigma_t^{\vphi|_F}$ for all $t\in \R$. Thus we can apply
Theorem~\ref{thm:unital_but_infinite_Kraus_approximation} to $F$, $T$, and $\frac{\delta}{4}$ to obtain $(a_i)_{i\in \N}\subset N^\vphi$ satisfying (iii) in Claim 2 and
    \[
        \left\|T(x) - \sum_{i=1}^\infty a_i^* x a_i \right\| < \frac{\delta}{4} \|x\|
    \]
for all $x\in F$. Using this as well as the proximity of $T(v_k)$ and $u_k$ given by Claim 1 we have
    \[
        \left\|u_k - \sum_{i=1}^\infty a_i^* v_k a_i \right\|_\vphi \leq \|u_k - T(v_k)\|_\vphi + \left\| T(v_k) - \sum_{i=1}^\infty a_i^* v_k a_i\right\|_\vphi < \frac{\delta}{2}.
    \]
We also compute
    \[
        \sum_{i=1}^\infty \|a_i u_k\|_\vphi^2 = \sum_{i=1}^\infty \vphi(u_k^* a_i^* a_i u_k)= \vphi(u_k^* u_k) = 1,
    \]
and
    \[
        \sum_{i=1}^\infty \|v_k a_i\|_\vphi^2 = \sum_{i=1}^\infty \vphi(a_i^* a_i) = 1.
    \]
Putting all of this together we obtain
    \begin{align*}
        \sum_{i=1}^\infty \|v_ka_i - a_k u_k\|_\vphi^2 &= 2 - 2 \sum_{i=1}^\infty \Re \vphi(u_k^* a_i^* v_k a_i)\\
        &= 2 \Re \vphi\left(u_k^*\left( u_k - \sum_{i=1}^\infty a_i^* v_k a_i \right) \right) \leq 2 \left\| u_k - \sum_{i=1}^\infty a_i^* v_k a_i \right\|_\vphi < \delta.
    \end{align*}
This proves (iv) and hence all of Claim 2.\\

{\leftskip 15pt
\noindent\textbf{Claim 3:} For each finite set $x_1,\ldots, x_n\in N$, finite set $E_0\subset \Sd(\vphi)$, and $\epsilon>0$ there exists a finite dimensional subfactor $F\leq N$ and a finite set $y_1,\ldots, y_n\in F$ satisfying:
    \begin{enumerate}[leftmargin=5em, label=(\roman*)]
        \item $\vphi=\vphi|_F\otimes \vphi|_{F^c}$;

        \item $(F,\vphi|_F)\cong (\mathbb{M}_E^{\otimes m}, \phi_E^{\otimes m})$ for some finite set $E_0\subset E\subset \Sd(\vphi)$ and $m\in \N$;

        \item $\| x_k - y_k\|_\vphi < \epsilon$ for each $k=1,\ldots, n$.
    \end{enumerate}
    
}

\hfill

\noindent Decomposing the $x_i$ into a linear combintation of four unitaries, it suffices to consider $u_1,\ldots, u_n\in \mathcal{U}(N)$. Apply Claim 2 to $u_1,\ldots, u_n$, $E_0$, and some $\delta>0$ to be determined later, and let $G\leq N$ and $(a_i)_{i\in \N}\subset  N^\vphi$ be the resulting data. The $N$-$N$-bimodule structure of $L^2(N,\vphi)$ restricts to an $N^\vphi$-$N^\vphi$-bimodule structure satisfying
    \[
        a\cdot x \cdot b = a J_\vphi b^* J_\vphi x= axb
    \]
for $a,b\in N^\vphi$ and $x\in N$. (Indeed, first check this for $x\in N^{(\vphi,\eig)}$ for which $J_\vphi x = \sigma_{-i/2}^{\vphi}(x^*)$.) Claim 2 therefore tells us that the tuples $(v_1,\ldots, v_n), (u_1,\ldots, u_n)\in L^2(N,\vphi)^{\oplus n}$ are $\delta$-related over $N^\vphi$ in the sense of \cite[Definition 2.1]{Haa89}. So given $\epsilon>0$, if we take $\delta>0$ sufficiently small then \cite[Theorem 2.3]{Haa89} yields a unitary $w\in \mathcal{U}(N^\vphi)$ satisfying
    \[
        \| u_k - w^* v_k w\|_\vphi = \| w u_k - v_k w\|_\vphi <\epsilon
    \]
for each $k=1,\ldots, n$. Define $F:= w^* G w$ and $y_k= w^* v_k w$ for each $k=1,\ldots,n$ so that (iii) holds. Since $w$ is in the centralizer for $\vphi$, we have 
    \[
        \vphi|_F\otimes \vphi|_{F^c} = (\vphi|_G\otimes \vphi|_{G^c})\circ( \Ad{w^*}\otimes \Ad{w^*}) = \vphi|_G\otimes \vphi|_{G^c}=\vphi,
    \]
so that (i) holds. Also
    \[
        (F,\vphi|_F)\cong (G,\vphi|_G)\cong (\mathbb{M}_E^{\otimes m}, \phi_E^{\otimes m}),
    \]
for some finite set $E_0\subset E\subset \Sd(\vphi)$ and $m\in \N$ so that (ii) and hence Claim 3 holds.\\

We now begin the proof of our main theorem in earnest. Using separability of $N$ we write a countable generating set $\{x_n\in N\colon n\in \N\}$ for $N$. Also enumerate $\Sd(\vphi)\cap (0,1)=\{\lambda_n\colon n\in \N\}$. We will inductively construct a sequence of commuting finite dimensional subfactors $(F_n)_{n\in \N}$ of $N$ and a sequence of finite subsets $(E_n)_{n\in \N}$ of $\Sd(\vphi)$ satisfying:
    \begin{enumerate}[label=(\alph*)]
        \item $\inf \{ \|x_k - y\|_\vphi \colon y\in F_1\vee \cdots \vee F_n\} < \frac1n$ for each $n\in \N$ and each $k=1,\ldots, n$;

        \item $\vphi= \vphi|_{F_n}\otimes \vphi|_{F_n^c}$ for each $n\in \N$;

        \item $\{\lambda_1,\ldots, \lambda_n\} \subset E_n \subset E_{n+1}$;

        \item for each $n\in \N$ there exists $m_n\in \N$ so that
            \[
                (F_n,\vphi|_{F_n})\cong (\mathbb{M}_{E_n}^{\otimes m_n}, \phi_{E_n}^{\otimes m_n}).
            \]
    
    \end{enumerate}
$F_1$ and $E_1$ are obtained by applying Claim 3 to $x_1$, $E_0=\{1\}$, and $\epsilon=1$. Suppose we have constructed $F_1,\ldots, F_n$ as above. Set
    \[
        F:=F_1\vee \cdots \vee F_n = \bigotimes_{k=1}^n F_k.
    \]
Then $F$ is $\sigma^\vphi$-invariant since (b) implies this is true for each $F_1,\ldots, F_n$, and hence $\vphi= \vphi|_F \otimes \vphi|_{F^c}$. Fix a family of matrix units $\{e_{i,j}\in F\colon 1\leq i,j\leq d\}$ so that for each $k=1,\ldots, n+1$
    \[
        x_k = \sum_{i,j=1}^d e_{i,j} x^{(k)}_{i,j}
    \]
where $x^{(k)}_{i,j}\in F^c$. Set
    \[
        K:= \sum_{i,j=1}^d \|e_{i,j}\|_\vphi.
    \]
Recall that $\vphi|_{F^c}$ is extremal almost periodic with $\Sd(\vphi|_{F^c})=\Sd(\vphi)$ by Lemma~\ref{lem:corners_of_ext_ap_weights_are_ext_ap}, so we can apply Claim 3 to $(F^c,\vphi|_{F^c})$, the finite set $\{x^{(k)}_{i,j}\colon 1\leq k\leq n+1,\ 1\leq i,j\leq d\}$, the finite set $\{\lambda_{n+1}\}\cup E_n$, and $\epsilon=\frac{1}{(n+1)K}$ to obtain: a finite subset $E_{n+1}\subset \Sd(\vphi)$ satisfying (c); a finite dimensional subfactor $F_{n+1}\leq F^c$ satisfying $\vphi|_{F^c} = \vphi|_{F_{n+1}} \otimes \vphi|_{F_{n+1}'\cap F^c}$ as well as (d); and a finite set $\{y^{(k)}_{i,j}\in F_{n+1}\colon 1\leq k\leq n+1\, 1\leq i,j\leq d\}$ satisfying
    \[
        \| x_{i,j}^{(k)} - y_{i,j}^{(k)} \|_\vphi < \frac{1}{(n+1)K^{1/2}}
    \]
for each $k=1,\ldots, n+1$ and $1\leq i,j\leq d$. Note that $F_{n+1}$ being a subfactor of $F^c$ implies it commutes with $F_1,\ldots, F_n$. Composing the expectations from $N$ to $F^c$ and from $F^c$ to $F_{n+1}$ gives (b). Towards establishing (a), set
    \[
        y_k:= \sum_{i,j=1}^d e_{i,j} y_{i,j}^{(k)}
    \]
for each $k=1,\ldots, n+1$ so that $y_k\in F\vee F_{n+1} = F_1\vee \cdots \vee F_{n+1}$ and using (b) we have
    \begin{align*}
        \| x_k - y_k\|_\vphi^2 &= \vphi\left( \sum_{i,j=1}^d e_{i,j} \sum_{r=1^d} (x_{\ell, i}^{(k)} - y_{\ell,i}^{(k)})^* (x_{\ell, j}^{(k)} - y_{\ell,j}^{(k)})  \right) \\
        &= \sum_{i,j,r=1}^d \vphi(e_{i,j})\vphi( (x_{\ell, i}^{(k)} - y_{\ell,i}^{(k)})^* (x_{\ell, j}^{(k)} - y_{\ell,j}^{(k)}) )\\
        &\leq \sum_{i,j=1} \|e_{i,j}\|_\vphi \frac{1}{(n+1)^2 K} \leq \frac{1}{(n+1)^2},
    \end{align*}
and so (a) follows. Induction therefore gives us the claimed sequence $(F_n)_{n\in\N}$ of commuting subfactors $N$ and the sequence $(E_n)_{n\in \N}$ of finite subsets of $\Sd(\vphi)$.

Now, if $\E_n$ is the $\vphi$=preserving conditional expectation onto $F_1\vee \cdots \vee F_n$ then for each $k\leq n$ we have
    \[
        \| x_k - \E_n(x_k)\|_\vphi \leq \frac{1}{n}
    \]
by (a). That is, the bounded sequence $(\E_n(x_k))_{n\in \N}$ converges to $x_k$ in $\|\,\cdot\,\|_\vphi$-norm and therefore in the strong operator topology. Consequently,
    \[
        \bigvee_{n\in \N} F_n = \left( \bigcup_{n\in \N} F_n\right)''
    \]
contains $x_k$ for all $k\in \N$. Since $\{x_k\colon k\in \N\}$ is a generating set for $N$, we see that the above must in fact equal $N$. Further, (b) implies $\vphi$ restricted to $F_1\vee \cdots \vee F_n$ conincides with
    \[
        \bigotimes_{k=1}^n \vphi|_{F_k}
    \]
for each $n\in \N$. Consequently, (c) implies
    \[
        (N,\vphi) \cong \bigotimes_{n\in \N} (F_n, \vphi|_{F_n}) \cong \bigotimes_{n\in \N} (\mathbb{M}_{E_n}^{\otimes m_n}, \phi_{E_n}^{\otimes m_n}).
    \]
The inclusions $\{\lambda_1,\ldots, \lambda_{n+1}\}\subset E_n \subset E_{n+1}$ implies that $(M_2(\C), \phi_{\lambda_n})$ appears infinite often in the above tensor product for each $n\in \N$. Thus
    \[
        (N,\vphi)\cong \bigotimes_{\Lambda\cap (0,1)} (R_\lambda, \vphi_\lambda) = (R_{\Lambda}, \vphi_{\Lambda}),
    \]
as claimed.$\hfill\square$

\section{Proof of Theorem~\ref{thm:C}}\label{section:Cuntz_algebras}

Using either \cite[Theorem 4.7]{Izu93} when $J$ is finite or \cite[Example 1]{Tho19} when $J$ is infinte, it follows that $\pi_\vphi(\mathcal{O}_J)''$ is a separable injective type $\mathrm{III}$ factor. Thus by Theorem~\ref{thm:A} it suffices to show $\vphi$ is extremal almost periodic with $\Sd(\vphi)=\Lambda$. The generators of $\mathcal{O}_J$ being eigenoperators with respect to $\vphi$ implies the almost periodicity of $\vphi$, and moreover that $\Sd(\vphi)\subset \Lambda$. Since $\Sd(\vphi)$ contains generators of the group $\Lambda$, we will obtain $\Sd(\vphi)=\Lambda$ once we know the former is a group. Therefore the proof will be complete once we have argued that $\vphi$ is extremal, and this is essentially already contained in the proof of \cite[Theorem 5]{Tho19}.

Indeed, \cite[Theorem 5]{Tho19} asserts that the \emph{Connes spectrum} (see \cite[Section XI.2]{Tak03}) of the modular automorphism group of $\sigma^\vphi$ is given by
    \[
        \Gamma(\sigma^\vphi) = \overline{\Lambda}.
    \]
In fact, for each projection $q$ in the center of $(\pi_\vphi(\mathcal{O}_J)'')^\vphi$ and each $j\in J$, in the proof of \cite[Theorem 5]{Tho19} Thomsen exhibits a partial isometry $w\in (\pi_\vphi(\mathcal{O}_J)'')^{(\vphi,\mu(j))}$ with $q wq\neq 0$. It follows that
    \[
        \mu(J)\subset \Sd(\vphi|_{ q \pi_\vphi(\mathcal{O}_J)'' q})
    \]
for all such $q$. Equip $\Lambda$ with the discrete topology, let $G$ be its compact dual, and $G\overset{\alpha}{\curvearrowright} \pi_\vphi(\mathcal{O}_J)''$ be the point modular extension of the modular automorphism group of $\vphi$. Then the above, in conjunction with \cite[Equation (1)]{Con74}, shows that
    \[
        \mu(J) \subset \text{Sp}( \alpha^{(q)}),
    \]
where $\alpha^{(q)}$ is the restriction of $\alpha$ to $qMq$ and $\text{Sp}(\alpha^{(q)})$ is its \emph{Arveson spectrum} (see \cite[Section XI.1]{Tak03}). Using \cite[Proposition 2.2.2]{Con73} we see that
    \[
        \mu(J)\subset \bigcap_{q} \text{Sp}( \alpha^{(q)})= \Gamma(\alpha).
    \]
But since the Connes spectrum of $\alpha$ is a subgroup of $\Lambda$ it follows that $\Gamma(\alpha)=\Lambda$. This in turn implies $(\pi_\vphi(\mathcal{O}_J)'')^\alpha = (\pi_\vphi(\mathcal{O}_J)'')^\vphi$ is a factor by \cite[Theorem 2.4.1]{Con73}. That is, $\vphi$ is extremal. $\hfill\square$

\bibliographystyle{amsalpha}
\bibliography{references}

\end{document}